\documentclass[reqno,11pt]{amsart}
\usepackage{amsmath,amsthm,amssymb}
\usepackage{latexsym,lineno,hyperref}
\newtheorem{theorem}{Theorem}[section]
\newtheorem{lemma}[theorem]{Lemma}

\theoremstyle{definition}

\theoremstyle{remark}

\numberwithin{equation}{section}
\newcommand{\rb}{\right}
\newcommand{\lb}{\left}
\newcommand\floor[1]{\lfloor#1\rfloor}

\begin{document}

\title[On Eight Colour Partitions]{On Eight Colour Partitions}

\author{B. Hemanthkumar}
\address{Department of Mathematics, M. S. Ramaiah University of Applied Sciences, Peenya campus, Bengaluru-560 058, Karnataka, India}
\email{hemanthkumarb.30@gmail.com}
\author{ H. S. Sumanth Bharadwaj}
\address{Department of Mathematics, M. S. Ramaiah University of Applied Sciences, Peenya campus, Bengaluru-560 058, Karnataka, India}
\email{sumanthbharadwaj@gmail.com}


\date{}
          
\begin{abstract}
In this article, we study the arithmetic properties of the partition function $p_{8}(n)$, the number of $8$-colour partitions of $n$. We prove several Ramanujan type congruences modulo higher powers of 2 for the function $p_8(n)$ by finding explicit formulas for the generating functions.
\end{abstract}

\maketitle

\section{Introduction}\label{S1}
For any nonnegative integer $n$, let $p(n)$ denote the number of partitions of $n$. The generating function of $p(n)$ is given by 
\begin{equation*}
\sum\limits_{n=0}^{\infty} p(n) q^n =\frac{1}{f_1}.
\end{equation*}
Here and throughout this paper, we set 
\[f_k:=\prod\limits_{j=0}^{\infty}(1-q^{kj})\]
for any positive integer $k$.
By Euler's pentagonal number theorem \cite[pp. 10-12]{And}, we have
\begin{equation}
f_1=\sum_{k=-\infty}^\infty (-1)^k q^{\frac{k(3k-1)}{2}}.\label{E4}
\end{equation}
In 1919, Ramanujan \cite{R1919} obtained the generating function for $p(5n+4)$ as 
\begin{equation}
\sum\limits_{n=0}^{\infty} p(5n+4) q^n =5 \frac{f_5^5}{f_1^6}. \label{E5}
\end{equation}
It follows from \eqref{E5} that $p(5n+4)\equiv 0 \pmod{5}$.  Also, he conjectured that for any integer $k\geq 1,$
\[p(5^kn+\delta_k) \equiv 0 \pmod{5^k},\]
 where $\delta_k$ is the reciprocal modulo $5^k$ of 24, and outlined the proof in \cite[pp.\ 156-177]{R}, (also see \cite{BK}). In 1938, Watson \cite{W} proved the above conjecture and in 1981,  Hirschhorn and Hunt \cite{HH} gave a simple proof of the conjecture.

Recently, Hirschhorn \cite{H} studied the number of 3-colour partitions of $n$, denoted by $p_3(n)$ and is given by  
\begin{equation*}
\sum\limits_{n=0}^{\infty} p_3(n) q^n = \dfrac{1}{f_1^3}.
\end{equation*}
He obtained several congruences for $p_3(n)$ modulo high powers of $3$, namely, for all $\alpha, n\geq 0$,
\begin{equation*}
p_3\lb(3^{2\alpha+1}n+\frac{5\cdot3^{2\alpha+1}+1}{8}\rb) \equiv 0 \pmod{3^{2\alpha+2}},
\end{equation*}
which are analogous to Ramanujan's congruences for the partition function. 

In this paper we study the partition function,
 $p_8(n)$, the number of 8-colour partitions of $n$, which satisfies
 \begin{equation}\label{E001}
 \sum\limits_{n=0}^{\infty}p_8(n)q^n = \frac{1}{(q;q)_{\infty}^8}=\frac{1}{f_1^8}.
 \end{equation}
We obtain appropriate generating formulae for $p_8(n)$ and deduce several Ramanujan-type congruences modulo high powers of 2. The main results of this paper are  as follows.

\begin{theorem}\label{T4}
	For all nonnegative integers $n$ and $\alpha\geq 1$, we have
	\begin{align}
	p_8(2n+1) & \equiv 0 \pmod{2^3}, \label{E094} \\
	p_8\lb(2^{2\alpha}n + \frac{2^{2\alpha+1}+1}{3} \rb) & \equiv 0 \pmod{2^{3\alpha+3}}, \label{E095}\\
	p_8\lb(2^{2\alpha+1}n + \frac{7\cdot 2^{2\alpha-1}+1}{3} \rb) & \equiv 0 \pmod{2^{3\alpha+2}}, \label{E096}\\
	p_8\lb(2^{2\alpha+1}n + \frac{5\cdot 2^{2\alpha}+1}{3} \rb) & \equiv 0 \pmod{2^{3\alpha+8}}, \label{E097}\\
	p_8\lb(2^{2\alpha+2}n + \frac{13\cdot 2^{2\alpha-1}+1}{3} \rb) & \equiv 0 \pmod{2^{3\alpha+1}}, \label{E098}\\
	p_8\lb(2^{2\alpha+2}n + \frac{19\cdot 2^{2\alpha-1}+1}{3} \rb)& \equiv 0 \pmod{2^{3\alpha+3}}, \label{E099}\\
	p_8\lb(2^{2\alpha+2}n + \frac{11\cdot 2^{2\alpha}+1}{3} \rb) & \equiv 0 \pmod{2^{3\alpha+10}}, \label{E100}\\
	p_8\lb(2^{2\alpha+3}n + \frac{17\cdot 2^{2\alpha}+1}{3} \rb) & \equiv 0 \pmod{2^{3\alpha+9}} \label{E101}
	\end{align} and if $n$ is not a generalised pentagonal number, then
	\begin{equation}
	p_8\lb(2^{2\alpha+2}n + \frac{2^{2\alpha-1}+1}{3} \rb) \equiv 0 \pmod{2^{3\alpha+1}}. \label{E102}
	\end{equation}
\end{theorem}
 
We provide some definitions and preliminary results in Section \ref{S2}. We establish generating functions for $p_8(n)$ in Section \ref{S3} and prove Theorem \ref{T4} in Section \ref{S4}.

\section{Preliminaries} \label{S2}
In this section, we present some preliminary results which are used in proving our main results.
 
\begin{lemma}\label{L0}
We have
\begin{equation}
f_1^4 = \frac{f_4^{10}}{f_2^2f_8^4} - 4q \frac{f_2^2f_8^4}{f_4^2} \label{E6}
\end{equation} and
\begin{equation}
\frac{1}{f_1^4} = \frac{f_4^{14}}{f_2^{14}f_8^4} + 4q \frac{f_4^2f_8^4}{f_2^{10}}. \label{E7}
\end{equation}
\end{lemma}
\begin{proof}
Lemma \ref{L0} is an immediate consequence of dissection formulas of Ramanujan, collected in Berndt’s book \cite[Entry 25, p.\ 40]{B}. 
\end{proof}
\begin{lemma}\label{L1}
Let
\[S=\frac{qf_4^8}{f_1^8}, \,\, T=\frac{q^2f_4^{24}}{f_2^{24}}.\]
Then, for any $j\geq 1$
\begin{equation*}
S^j=T(S^{j-2}+16S^{j-1}). 
\end{equation*}
\end{lemma}
\begin{proof}
Squaring \eqref{E6} on both sides, we deduce that
\begin{equation}\label{E1}
f_1^8=\frac{f_4^{20}}{f_2^4f_8^8} + 16q^2 \frac{f_2^{4}f_8^8}{f_4^4}-8qf_4^8.
\end{equation}
Setting $-q$ for $q$ in \eqref{E1} and using the fact that $\prod\limits_{j=0}^{\infty}(1+q^{j})=\dfrac{f_2^3}{f_1f_4}$, we obtain
\begin{equation}\label{E11}
\frac{f_2^{24}}{f_1^8f_4^8}=\frac{f_4^{20}}{f_2^4f_8^8} + 16q^2 \frac{f_2^{4}f_8^8}{f_4^4}+8q f_4^{8}.
\end{equation}
We can rewrite \eqref{E1} as
\begin{equation}\label{E12}
\frac{1}{S}=\frac{f_4^{12}}{qf_2^4f_8^8} + 16q \frac{f_2^{4}f_8^8}{f_4^{12}}-8
\end{equation}
and then \eqref{E11} as
\begin{equation*}
\frac{1}{S}+16=\frac{f_2^{24}}{qf_4^{16}f_1^8}.
\end{equation*}
If follows from the definition of $S$ and last equation that
\begin{equation*}
\frac{1}{S}\left(\frac{1}{S}+16\right)=\frac{f_2^{24}}{q^2f_4^{24}}=\frac{1}{T},
\end{equation*}
which yields
\begin{equation*}
S=T\lb(\frac{1}{S}+16\rb),
\end{equation*}
thus for any $j \geq 1$,
\begin{equation*}
S^j = T\lb(S^{j-2}+16S^{j-1}\rb).
\end{equation*}
This completes the proof of Lemma \ref{L1}.
\end{proof}
Let us define an operator $H$ as follows:
\begin{equation*}
H\lb (\sum\limits_{n=0}^{\infty} a_nq^n\rb )=\sum\limits_{n=0}^{\infty} a_{2n}q^{2n} \label{E011}
\end{equation*}
and $H(1)=1.$

In view of \eqref{E12}, we see that
\begin{equation*}
H\lb(\frac{1}{S}\rb)=-8. 
\end{equation*} 
Also, from Lemma \ref{L1}
\begin{equation}
H\lb(S^j\rb)=T\lb(H\lb(S^{j-2}\rb)+16H\lb(S^{j-1}\rb)\rb ). \label{E029} 
\end{equation}
In particular,
\begin{align}
H(S) &= T\lb(H\lb(\dfrac{1}{S}\rb)+ 16 H(1) \rb) =8T, \label{E030}\\
H(S^2) &= T\lb(H(1)+16 H(S)\rb) =T+128T^2 \label{E031}
\end{align} and 
\begin{equation}
H(S^3) = T\lb(H(S)+16 H(S^2)\rb) =24T^2+2048T^3. \label{E034}
\end{equation}

We define an infinite matrix $M=(m_{j,\,k})_{j,\,k\ge1}$ as follows:
\begin{enumerate}
\item \label{MR1} $m_{1\,,1}=8,\ m_{1,\,k}=0$ for all $k \geq 2$.
\item \label{MR2} $m_{2,\,1}=1,\ m_{2,\,2}=128,\ m_{2,\,k}=0,$ for all $k \geq 3$.
\item \label{MR3} $m_{j,\,1}=0$ for all $j \geq 3$.
\item \label{MR4} $m_{j,\,k}=16m_{j-1,\,k-1}+m_{j-2,\,k-1}$ for all $j\geq3,\ k\geq 2$.
\end{enumerate}
\noindent Using the definition of $M$ and induction, it is easy to prove that 
\begin{enumerate}
\setcounter{enumi}{4}
\item \label{MR5} $m_{j\,,k}=0$ for all $k>j$.
\item \label{MR6} $m_{j,\,k}=0$ for all $j>2k$.
\item \label{MR7} $m_{2j,\,j}=1$ for all $j\geq 1$.
\item \label{MR8} $m_{j,\,j}=2^{4j-1}$ for all $j\geq 1$.
\end{enumerate}
We omit the details of the proof. The first eight rows of $M$ are given by
\begin{equation*}\label{E032}
\lb (\begin{matrix} 2^3&0&0&0&0&0&0&0&\cdots\\
1&2^7&0&0&0&0&0&0&\cdots\\
0&2^3\times 3&2^{11}&0&0&0&0&0&\cdots\\
0&1&2^9&2^{15}&0&0&0&0&\cdots\\
0&0&2^3\times 5&2^{11}\times 5&2^{19}&0&0&0&\cdots\\
0&0&1&2^7\times 3^2&2^{16}\times 3&2^{23}&0&0&\cdots\\
0&0&0&2^3\times 7&2^{12}\times 7&2^{19}\times 7&2^{27}&0&\cdots\\
0&0&0&1&2^{11}&2^{17}\times 5&2^{26}&2^{31}&\cdots
\end{matrix}\rb ).\\
\end{equation*}

\begin{lemma} \label{L3}
For any positive integer $j$, we have
\begin{equation}
H\lb(S^{j}\rb)=\sum\limits_{k=1}^{j} m_{j,\,k}T^k = \sum\limits_{k=\floor{\frac{j+1}{2}}}^{j} m_{j,\,k}T^k. \label{E039}
\end{equation}
\end{lemma}
\begin{proof}
The second equality follows from \eqref{MR6}. We use mathematical induction to prove the first equality. From \eqref{E030} and \eqref{E031}, we see that the above identity holds for $j=1,2$. Suppose that \eqref{E039} holds for some integer $j>2$. From \eqref{E029}, we see that 
\begin{align*}
H\lb(S^{j+1}\rb) & = T\lb( \sum\limits_{k=1}^{j-1} m_{j-1,\,k}T^k + 16 \sum\limits_{k=1}^{j} m_{j,\,k}T^k\rb)\\
& = 16m_{j,\,j} T^{j+1} + \sum\limits_{k=2}^{j}\lb(16 m_{j,\,k-1} + m_{j-1,\,k-1}\rb)T^{k}. \label{E040}
\end{align*}
Using \eqref{MR3},\eqref{MR4} and \eqref{MR8} in the above equation, we arrive at 
\begin{equation*}
H\lb(S^{j+1}\rb)=\sum\limits_{k=1}^{j+1} m_{j+1,\,k}T^k.\label{E041}
\end{equation*}
This completes the induction.
\end{proof}

\section{Generating functions} \label{S3}
In this section we obtain generating functions for the sequences in Theorem \ref{T4}.
\begin{theorem}\label{L2}
	We have
	\begin{align}
	\sum\limits_{n=0}^{\infty}p_8(2n+1)q^n &= 8\frac{f_2^{16}}{f_1^{24}}, \label{E020}\\
	\sum\limits_{n=0}^{\infty}p_8(4n+3)q^n &= 192\frac{f_2^{24}}{f_1^{32}} + 16384 q \frac{f_2^{48}}{f_1^{56}}.  \label{E063}
	\end{align} 
\end{theorem}
\begin{proof}
	We note that from \eqref{E030},
	\begin{equation*}
	H\lb(q\frac{f_4^8}{f_1^8}\rb)=8q^2\frac{f_4^{24}}{f_2^{24}}
	\end{equation*}
	or
	\begin{equation*}
	H\lb(\frac{q}{f_1^8}\rb)=8q^2\frac{f_4^{16}}{f_2^{24}}.
	\end{equation*}
	This proves \eqref{E020}. Rewriting \eqref{E034}, we see that 
	\begin{equation*}
	H\lb(q^3\frac{f_4^{24}}{f_1^{24}}\rb)=24q^4\frac{f_4^{48}}{f_2^{48}}+2048q^6\frac{f_4^{72}}{f_2^{72}}
	\end{equation*}
	or
	\begin{equation}\label{E34}
	H\lb(q\frac{f_2^{16}}{f_1^{24}}\rb)=24q^2\frac{f_4^{24}}{f_2^{32}}+2048q^4\frac{f_4^{48}}{f_2^{56}}.
	\end{equation}
	Thus, \eqref{E063} follows from \eqref{E020} and \eqref{E34}.
\end{proof}
Let us define another infinite matrix $(x_{\alpha,\, i})_{\alpha, j\geq1}$ by
 \begin{enumerate}
\setcounter{enumi}{8}
\item \label{MR9} $x_{1,\,1}=8, \  x_{1,\,k}=0$ for all $k\geq2$.
\item \label{MR10} \[x_{\alpha+1, \, j} = \begin{cases}
\displaystyle \sum_{i=1}^{\infty} x_{\alpha,\,i} \, m_{3i,\,i+j} & \text{if}\, \, \, \alpha \ \text{is odd}, \vspace{.1cm} \\ 
\displaystyle \sum_{i=1}^{\infty} x_{\alpha,\,i} \, m_{3i+1,\,i+j} & \text{if}\, \, \, \alpha \ \text{is even},
 \end{cases}\] 
 \end{enumerate}
Using the properties of the numbers $m_{i,\,j}$ and induction, we can easily see that
\begin{enumerate}
\setcounter{enumi}{10}
\item \label{MR11} \[x_{\alpha, \, j} = 0 \,\,\,\,  \text{if} \,\ \begin{cases}
j>\frac{2^{\alpha+1}-2}{3}  & \text{and} \,\ \alpha \,\ \text{is even}, \vspace{.1cm}\\
j>\frac{2^{\alpha+1}-1}{3} & \text{and}\,\ \alpha \,\, \text{is odd};
\end{cases}\]
\end{enumerate}
and 
\begin{enumerate}
\setcounter{enumi}{11}
\item \label{MR12} 
\begin{align*}
x_{\alpha,\, \frac{2^{\alpha+1}-2}{3}}&=2^{8(2^\alpha-1)-5\alpha} \,\,\ \text{if} \,\ \alpha \,\ \text{is even},\\
x_{\alpha,\, \frac{2^{\alpha+1}-1}{3}}&=2^{8(2^\alpha-1)-5\alpha} \,\ \text{if} \,\,\ \alpha \,\ \text{is odd}.
\end{align*}
\end{enumerate}
The first four rows of $x$ are given by
\begin{equation}\label{31}
\left (\begin{matrix}
2^3&0&0&0&0&\cdots\\
2^6\times 3&2^{14}&0&0&0&\cdots\\
2^6\times 3&2^{15}\times 31&2^{21}\times 227&2^{33}\times 7 &2^{41}&\cdots\\
2^9\times 1993&2^{17}\times 729187&2^{31}\times 265617&2^{38}\times 3070947&\cdots&\cdots
\end{matrix}\right ).\\
\end{equation}

\begin{theorem}\label{T1}
	For any positive integer $\alpha$,
	\begin{equation}
		\sum\limits_{n=0}^{\infty} p_8\lb(2^{2\alpha-1}n + \frac{2^{2\alpha-1}+1}{3} \rb)q^{n+1}
		=\frac{1}{f_2^8}\sum_{j=1}^{\frac{1}{3}(4^{\alpha}-1)}x_{2\alpha-1,\,j}\, q^j \frac{f_2^{24j}}{f_1^{24j}}\label{E004}
	\end{equation}
	and 
	\begin{equation}
		\sum\limits_{n=0}^{\infty} p_8\lb(2^{2\alpha}n + \frac{2^{2\alpha+1}+1}{3} \rb)q^{n+1}
		=\frac{1}{f_1^8}\sum_{j=1}^{\frac{2}{3}(4^{\alpha}-1)}x_{2\alpha,\,j}\, q^j \frac{f_2^{24j}}{f_1^{24j}}.\label{E005}
	\end{equation}
\end{theorem}
\begin{proof}
We prove \eqref{E004} and \eqref{E005} by mathematical induction on $\alpha$. From \eqref{E020} and \eqref{E063}, it follows that \eqref{E004} and \eqref{E005} hold for $\alpha=1$.

Suppose that \eqref{E004} holds for some positive integer $\alpha>1$. 
Applying the operator $H$ to both sides gives
\begin{align}
\nonumber & \sum\limits_{n=0}^{\infty} p_8\lb(2^{2\alpha-1}(2n+1) + \frac{2^{2\alpha-1}+1}{3} \rb)q^{2n+2}\\
&=\frac{1}{f_2^8}\sum_{j=1}^{\frac{1}{3}(4^{\alpha}-1)}x_{2\alpha-1,j} H\lb(q^j \frac{f_2^{24j}}{f_1^{24j}}\rb). \label{E053}
\end{align}
From Lemma \ref{L3}, we have
\begin{align}
\nonumber H\lb(q^j \frac{f_2^{24j}}{f_1^{24j}}\rb)&=\frac{f_2^{24j}}{q^{2j}f_4^{24j}}H\lb(q^{3j} \frac{f_4^{24j}}{f_1^{24j}}\rb)\\
\nonumber &=\frac{f_2^{24j}}{q^{2j}f_4^{24j}} \sum_{k=\floor{\frac{3j+1}{2}}}^{3j} m_{3j, \,k} \, q^{2k} \frac{f_4^{24k}}{f_2^{24k}}\\
&=\sum_{k=\floor{\frac{j+1}{2}}}^{2j}m_{3j,\,k+j} \,q^{2k}\frac{f_4^{24k}}{f_2^{24k}}.\label{E2}
\end{align}
Combining \eqref{E053} and \eqref{E2}, we obtain
\begin{align*}
&\sum\limits_{n=0}^{\infty} p_8\lb(2^{2\alpha-1}(2n+1) + \frac{2^{2\alpha-1}+1}{3} \rb)q^{2n+2}\\
&=\frac{1}{f_2^8}\sum_{j=1}^{\frac{1}{3}(4^{\alpha}-1)}\sum_{k=\floor{\frac{j+1}{2}}}^{2j}x_{2\alpha-1,\,j} \, m_{3j,\,k+j} \,q^{2k}\frac{f_4^{24k}}{f_2^{24k}}.
\end{align*}
Interchanging the order of summantion and using properties \eqref{MR5}, \eqref{MR6} and \eqref{MR11} to extend the sums to all positive integers, we have
\begin{align*}
&\sum\limits_{n=0}^{\infty} p_8\lb(2^{2\alpha}n + \frac{2^{2\alpha+1}+1}{3} \rb)q^{n+1}\\
\nonumber&=\frac{1}{f_1^8}\sum_{k=1}^{\infty}\lb(\sum_{j=1}^{\infty}x_{2\alpha-1,\,j} \, m_{3j,\,k+j}\rb) q^{2k}\frac{f_2^{24k}}{f_1^{24k}}\\
& = \frac{1}{f_1^8}\sum_{k=1}^{\frac{2}{3}(4^\alpha-1)}x_{2\alpha,\,k} \,q^{k}\frac{f_2^{24k}}{f_1^{24k}}
\end{align*}
which is \eqref{E005}. Here last equality follows from properties \eqref{MR10} and \eqref{MR11}.

Suppose \eqref{E005} holds for some positive integer $\alpha$. Multiplying the above equation by $q$ and applying the operator $H$ to the resulting expression, we find that
\begin{align}
\nonumber & \sum\limits_{n=0}^{\infty} p_8\lb(2^{2\alpha}(2n) + \frac{2^{2\alpha+1}+1}{3} \rb)q^{2n+2}\\
&=\sum_{j=1}^{\frac{2}{3}(4^{\alpha}-1)}x_{2\alpha,\,j} H\lb(q^{j+1} \frac{f_2^{24j}}{f_1^{24j+8}}\rb).\label{E055}
\end{align}
From Lemma \ref{L3}, we have
\begin{align}
\nonumber H\lb(q^{j+1} \frac{f_4^{24j+8}}{f_1^{24j+8}}\rb)&=\frac{f_2^{24j}}{q^{2j}f_4^{24j+8}}H\lb(q^{3j+1} \frac{f_4^{24j+8}}{f_1^{24j+8}}\rb)\\
\nonumber &=\frac{f_2^{24j}}{q^{2j}f_4^{24j+8}} \sum_{k=\floor{\frac{3j+2}{2}}}^{3j+1} m_{3j+1, \,k} \, q^{2k} \frac{f_4^{24k}}{f_2^{24k}}\\
&=\sum_{k=\floor{\frac{j+2}{2}}}^{2j+1}m_{3j+1,\,k+j} \,q^{2k}\frac{f_4^{24k-8}}{f_2^{24k}}.\label{E8}
\end{align}
Combine \eqref{E055} and \eqref{E8}, interchange the order of the summation and extend sums to all positive integers, we get
\begin{align*}
& \sum\limits_{n=0}^{\infty} p_8\lb(2^{2\alpha+1}n + \frac{2^{2\alpha+1}+1}{3} \rb)q^{n+1}\\
&=\frac{1}{f_2^8}\sum_{k=1}^{\infty}\lb(\sum_{j=1}^{\infty}x_{2\alpha,\,j} \, m_{3j+1,\,k+j}\rb) q^{k}\frac{f_2^{24k}}{f_1^{24k}}\\
&=\frac{1}{f_2^8}\sum_{k=1}^{\frac{1}{3}(4^{\alpha+1}-1)}x_{2\alpha+1,\,k}\, q^{k}\frac{f_2^{24k}}{f_1^{24k}}
\end{align*}
which is \eqref{E004} with $\alpha$ replaced by $\alpha+1$. This completes the proof of \eqref{E004} and \eqref{E005} by induction.
\end{proof}

\section{Congruences}\label{S4}
For a positive integer $n$, let $\vartheta_2(n)$ be the highest power of $2$ that divides $n$ and $\vartheta_2(0)=+\infty$ by convention.

In this section, we consider the powers of 2 that divide the numbers $m_{j,\,k}$ and $x_{j,\,k}$. Using this information we prove Theorem \ref{T4}.
\begin{lemma}\label{L4}
For any integers $j,k\geq 1$ and $k\leq j \leq 2k$, we have
\begin{equation}
\vartheta_2(m_{j,\,k}) \geq 4(2k-j)-1.  \label{E067}
\end{equation}
\end{lemma}
\begin{proof}
The proof easily follows from the definition of $m_{j,\,k}$ and properties \eqref{MR5}-\eqref{MR8}.
\end{proof}
\begin{lemma}\label{L5}
For all positive integers $j$ and $k$, we have
\begin{equation}\label{E068}
\vartheta_2(x_{2j-1,\,k}) \geq 3j+7(k-1) 
\end{equation} and
\begin{equation}\label{E069}
\vartheta_2(x_{2j,\,k}) \geq 3(j+1)+8(k-1). 
\end{equation}
In both cases, equality holds for $k=1$.
\end{lemma}
\begin{proof}
We use mathematical induction to prove these results, passing alternatively between \eqref{E068} and \eqref{E069}. From \eqref{31} it follows that \eqref{E068} and \eqref{E069} hold for $j=1,2$ and equality hold in each case for $k=1$.

Assume that \eqref{E068} is true for some positive integer $j$. Consider the case $k=1$ of \eqref{E069}.
\begin{align}
\nonumber x_{2j,\,1} &= \sum\limits_{i=1}^{\infty} x_{2j-1,\,i}\ m_{3i,\,i+1} = x_{2j-1,\,1}\ m_{3,2} + x_{2j-1,\,2}m_{6,\,3}\\
&=3\cdot 2^3 x_{2j-1,\,1} + x_{2j-1,\,2}. \label{E073}
\end{align}
By our assumption, $\vartheta_2(x_{2j-1,\,1})=3j$ and $\vartheta_2(x_{2j-1,\,2})\geq 3j+7$. Hence it follows that $\vartheta_2(x_{2j,\,1})=3j+3$. For $k\geq 2$, using the definition of $x_{\alpha, \,j}$ and properties \eqref{MR5} and \eqref{MR6} of the numbers $m_{j,\,k}$, we get 
\begin{equation}
\vartheta_2(x_{2j,\,k}) = \vartheta_2\lb(\sum\limits_{\frac{k}{2}\leq i\leq 2k} x_{2j-1,\,i}\ m_{3i,\,i+k}\rb)\geq \min\limits_{\frac{k}{2}\leq i\leq 2k} \vartheta_2(x_{2j-1,\,i}\ m_{3i,\,i+k}). \label{E075}
\end{equation}
By Lemma \ref{L4}, $\vartheta_2(m_{3i,\,i+k})\geq 4(2k-i)-1$, whenever $i \leq 2k$. Using this fact and the induction hypothesis,
\begin{align}
\nonumber \min\limits_{\frac{k}{2}\leq i\leq 2k} \vartheta_2(x_{2j-1,\,i}\  m_{3i,\,i+k}) &\geq \min\limits_{\frac{k}{2}\leq i\leq 2k} 3j+7(i-1) + 4(2k-i)-1\\ \nonumber  &= \min\limits_{\frac{k}{2}\leq i\leq 2k} 3j+3i+8(k-1)\\ &\geq 3j+3+8(k-1).\label{E076}
\end{align}
Combining \eqref{E075} and \eqref{E076}, we obtain
\begin{equation}
\vartheta_2(x_{2j,\,k})  \geq 3(j+1)+8(k-1). \label{E078}
\end{equation}
From \eqref{E078}, we conclude that if \eqref{E068} is true for some positive integer $j$ then \eqref{E069} is also true for that $j$. 

Suppose \eqref{E069} is true for some positive integer $j$. Consider
\begin{equation}
\nonumber x_{2j+1,\,1} = \sum\limits_{i=1}^{\infty} x_{2j,\,i}\ m_{3i+1,\,i+1} = x_{2j,\,1}\ m_{4,\,2} = x_{2j,\,1}. \label{E079}
\end{equation}
By our assumption, $\vartheta_2(x_{2j,1})=3(j+1)$ and hence $\vartheta_2(x_{2j+1,1})=3(j+1)$. Let $k\geq 2$, by the definition of $x_{\alpha,\,j}$ and properties \eqref{MR5} and \eqref{MR6}, we see that 
\begin{align*}
\vartheta_2(x_{2j+1,\,k}) & = \vartheta_2\lb(\sum\limits_{\frac{k-1}{2}\leq i\leq 2k-1} x_{2j,\,i}\ m_{3i+1,\,i+k}\rb)\\ &\geq \min\limits_{\frac{k-1}{2}\leq i\leq 2k-1} \vartheta_2(x_{2j,\,i}\ m_{3i+1,\,i+k})\\
 &\geq \min\limits_{\frac{k-1}{2}\leq i\leq 2k-1} 3(j+1)+8(i-1) + 4(2k-i-1)-1\\   
 &= \min\limits_{\frac{k-1}{2}\leq i\leq 2k-1} 3(j+1)+7(k-1)+(k+4i-6)\\ &\geq 3(j+1)+7(k-1).\label{E082}
\end{align*}
which is \eqref{E068} with $j$ replaced by $j+1$. Here second inequality of the last expression follows from the Lemma \ref{L4}.
\end{proof}

\begin{proof}[Proof of Theorem \ref{T4}]
Congruence \eqref{E094} follows from \eqref{E020}. By the binomial theorem, we can see that for all positive integers $k$ and $m$
\begin{equation}\label{E103}
f_k^{2^m} \equiv f_{2k}^{2^{m-1}} \pmod{2^{m}}. 
\end{equation}
In view of \eqref{E068} and \eqref{E103}, we see that
\begin{equation}\label{E10}
x_{2\alpha-1, \,k}\equiv 0 \pmod{2^{3\alpha+7(k-1)}}
\end{equation}
and 
\begin{equation}\label{E9}
\frac{f_2^{16}}{f_1^{24}}\equiv f_2^4 \pmod{2^3}
\end{equation}
It follows from \eqref{E004}, \eqref{E10} and \eqref{E9} that, 
\begin{equation*}
\sum\limits_{n=0}^{\infty} p_8\lb(2^{2\alpha-1}n + \frac{2^{2\alpha-1}+1}{3} \rb)q^{n}
\equiv  x_{2\alpha-1,\,1} \, f_2^4 \pmod{2^{3\alpha+3}},
\end{equation*}
which implies that
\begin{equation}\label{E13}
\sum\limits_{n=0}^{\infty} p_8\lb(2^{2\alpha}n + \frac{2^{2\alpha+1}+1}{3} \rb)q^{n}\equiv 0 \pmod{2^{3\alpha+3}}
\end{equation}
and
\begin{equation}\label{E14}
\sum\limits_{n=0}^{\infty} p_8\lb(2^{2\alpha}n + \frac{2^{2\alpha-1}+1}{3} \rb)q^{n} \equiv x_{2\alpha-1, \,1} \,f_1^4 \pmod{2^{3\alpha+3}}. 
\end{equation}
Congruence \eqref{E095} follows from \eqref{E13}. By substituting \eqref{E6} in \eqref{E14} and extracting the terms involving odd and even powers of $q$, 
\begin{align}
\nonumber \sum\limits_{n=0}^{\infty} p_8\lb(2^{2\alpha+1}n + \frac{7\cdot 2^{2\alpha-1}+1}{3} \rb)q^{n} & \equiv -4 x_{2\alpha-1,\,1}\, \frac{f_1^2f_{4}^4}{f_2^2} \\ & \equiv -4 x_{2\alpha-1,\,1}\, \frac{f_{4}^4}{f_2} \pmod{2^{3\alpha+3}} \label{E106}
\end{align} and
\begin{equation}
 \sum\limits_{n=0}^{\infty} p_8\lb(2^{2\alpha+1}n + \frac{2^{2\alpha-1}+1}{3} \rb)q^{n}  \equiv  x_{2\alpha-1,\,1} \, \frac{f_2^{10}}{f_1^2f_4^4} \equiv  x_{2\alpha-1,\,1}\, f_2 \pmod{2^{3\alpha+1}}. \label{E107}
\end{equation} 
Congruences \eqref{E096} and \eqref{E099} follow from \eqref{E106}. Congruences \eqref{E098} and \eqref{E102} follow from \eqref{E107}. 

From \eqref{E069}, we have
\begin{equation}\label{E15}
x_{2\alpha, \,k}\equiv 0 \pmod{2^{3\alpha+3+8(k-1)}}
\end{equation}
In view of \eqref{E005} and \eqref{E15},
\begin{equation}\label{E109}
\sum\limits_{n=0}^{\infty} p_8\lb(2^{2\alpha}n + \frac{2^{2\alpha+1}+1}{3} \rb)q^{n}
\equiv x_{2\alpha,1}\,  \frac{f_2^{24}}{f_1^{32}}  \pmod{2^{3\alpha+11}}. 
\end{equation}
Substituting \eqref{E7} in \eqref{E109} and extracting the terms involving odd powers of $q$, we see that
\begin{align}
\nonumber \sum\limits_{n=0}^{\infty} p_8\lb(2^{2\alpha+1}n + \frac{5\cdot 2^{2\alpha}+1}{3} \rb)q^{n}
& \equiv 2^5x_{2\alpha,1}\, \frac{f_2^{100}}{f_1^{84}f_4^{24}}\\ 
& \equiv 2^5x_{2\alpha,1}\, f_2^{10} \pmod{2^{3\alpha+10}}, \label{E110}
\end{align} which implies that
\begin{equation}
\sum\limits_{n=0}^{\infty} p_8\lb(2^{2\alpha+1}n + \frac{5\cdot 2^{2\alpha}+1}{3} \rb)q^{n}
 \equiv 2^5x_{2\alpha,1}\, f_4^{5} \pmod{2^{3\alpha+9}}. \label{E111}
\end{equation}
Congruences \eqref{E097} and \eqref{E100} follow from \eqref{E110}. Congruence \eqref{E101} follows from \eqref{E111}. 
\end{proof}

\end{document}